\documentclass[12pt,reqno]{amsart}

\usepackage[colorlinks=true,linkcolor=blue,citecolor=blue,urlcolor=blue]{hyperref}
\usepackage{array}
\usepackage{amssymb}
\usepackage{bm,bbding}
\usepackage{booktabs}
\usepackage{cancel} 
\usepackage{cleveref}
\usepackage{makecell}
\usepackage{ulem}
\usepackage{xcolor}
\usepackage{pstricks-add}
\usepackage[all]{xy}

\newtheorem{theorem}{Theorem}[section]
\newtheorem{lemma}[theorem]{Lemma}
\newtheorem{proposition}[theorem]{Proposition}
\newtheorem{corollary}[theorem]{Corollary}

\numberwithin{equation}{section}
\theoremstyle{definition}
\newtheorem{remark}[theorem]{Remark}

\newcommand{\cp}{\mathbb{C}P}
\newcommand{\cc}{\mathbb{C}}
\newcommand{\zz}{\mathbb{Z}}
\newcommand{\qq}{\mathbb{Q}}
\newcommand{\sign}{{\rm sign}}
\newcommand{\td}{{\rm td}}
\newcommand{\Td}{{\rm Td}}

\renewcommand\arraystretch{1.2}

\allowdisplaybreaks

\begin{document} 

\title{Elliptic genera of level $N$ for complete intersections}{\thanks{This project was supported by the Natural Science Foundation of Tianjin City of China, No. 19JCYBJC30300.}}

\author{Jianbo Wang, Yuyu Wang, Zhiwang Yu}
\address{(Jianbo Wang) School of Mathematics, Tianjin University, Tianjin 300350, China}
\email{wjianbo@tju.edu.cn}
\address{(Yuyu Wang) College of Mathematical Science, Tianjin Normal University, Tianjin 300387, China}
\email{wdoubleyu@aliyun.com}
\address{(Zhiwang Yu) School of Mathematics, Tianjin University, Tianjin 300350, China}
\email{yzhwang@tju.edu.cn}

\begin{abstract}
We study the elliptic genera of level $N$ at the cusps of $\Gamma_1(N)$ for any complete intersection. These genera are described as the summations of generalized binomial coefficients, where each generalized binomial coefficient is related to the dimension and multi-degree of complete intersection. For complete intersection $X_n(\underline{d})$, write $c_1(X_n(\underline{d}))=c_1x$, where $x\in H^2(X_n(\underline{d});\zz)\cong\zz$ is a generator. We mainly discuss the values of the elliptic genera of level $N$ for $X_n(\underline{d})$ in the case of $c_1>0, =0$ or $<0$. In particular, the values about the Todd genus, $\hat{A}$-genus and $A_k$-genus of $X_n(\underline{d})$ can be derived from the elliptic genera of level $N$.
\end{abstract}

\date{\today}
\maketitle

\section{Introduction}

It is well-known that the $\hat{A}$-genus is an obstruction to a non-trivial circle action on the given manifold. For example, Atiyah and Hirzebruch \cite{AtHi70} showed that if $M$ is a connected $2n$-dimensional spin manifold and $S^1$ acts non-trivially and smoothly on $M$, then the $\hat{A}$-genus of $M$ vanishes. There are also some other genera behaving just like the $\hat{A}$-genus as obstructions to the existence of 
circle actions on manifolds.

If a complex manifold with the first Chern class $c_1\equiv 0\pmod N$ admits a non-trivial holomorphic circle action, then the  elliptic genera of level $N$ at the cusps vanish (\cite{Hi1988}). Krichever showed that the $A_k$-genus ($k\geqslant 2$) vanishes if a unitary manifold admits a non-trivial $S^1$-action preserving its stably almost complex structure under the condition $c_1\equiv 0\pmod k$(\cite[Theorem 2.2]{Kricever1976}).

The Todd genus is considered as an obstruction to the existence of symplectic and Hamiltonian symplectic circle actions on the given manifold. Fel'dman showed that the Todd genus of a manifold admitting a symplectic circle action with isolated fixed points is equal either to 0, in which case the action is non-Hamiltonian, or to 1, in which case the action is Hamiltonian (\cite[Theorem 1.1]{Fel'dman2001}). Herrera showed that the Todd genus vanishes if there exists a holomorphic circle action on $\pi_2$-finite compact complex manifolds (\cite[Corollary 3.1]{Herrera2007}). 

For complete intersections, the $\hat{A}$-genus, Todd genus, $A_k$-genus and elliptic genera of level $N$ at the cusps of a congruence subgroup $\Gamma\subset SL_2(\zz)$ of finite index can be calculated. It is a natural question to detect the condition under which these genera vanish.

For any complete intersection $X_n(\underline{d}):=X_n(d_1,\dots,d_r)$ with $d_1,\dots,d_r>1$, let $c_1=n+r+1-\sum\limits_{i=1}^rd_i$, the main theorems of this paper are as follows:
\begin{theorem}\label{mainTgci}
The Todd genus of $X_n(\underline{d})$ satisfies the following properties:

{\rm (1)} ~$
\Td(X_n(\underline{d}))=
\begin{cases}
1, & \text{if~} c_1>0;\\
1+(-1)^n, & \text{if~} c_1=0.
\end{cases}$

{\rm (2)} If $c_1<0$, then $
(-1)^n \Td(X_n(\underline{d}))\geqslant n+r.$
\end{theorem}

Note that, by \cite[page 14]{Hi1978}, the Todd genus $\Td(M)$ for $2n$-dimensional manifold $M$ with $n$ odd is divisible by $c_1$. Hence, if $c_1=0$, then $\Td(M)$ vanishes.

For a compact complex manifold $M$, let $E$ be a holomorphic vector bundle over $M$, $K$ the canonical bundle over $M$, and let $\chi(M,E)$ be the Euler-Poincar\'e characteristic of $M$ (\cite[\S 15]{Hi1978}), then $\chi(M,K^{\frac{k}{N}})$  are the values of the elliptic genera of level $N$ at some cusps
of $\Gamma_1(N)$ for $0<k<N$ (\cite[\S 7.2]{HiBeJu1992}).
\begin{theorem}\label{Intro-mainNgci}
If $c_1\equiv 0\pmod N$, then for any integer $k$ with $0\leqslant k\leqslant N$, $\chi\left(X_n(\underline{d}),K^\frac{k}{N}\right)$ satisfy the following properties:

{\rm (1)} If $c_1>0$, then $
\chi\left(X_n(\underline{d}),K^\frac{k}{N}\right)=
\begin{cases}
1, & \text{if~} k=0;\\
0, & \text{if~} 0<k<N;\\
(-1)^n, & \text{if~} k=N.
\end{cases}$

{\rm (2)} If $c_1=0$, then $\chi\left(X_n(\underline{d}),K^\frac{k}{N}\right)=1+(-1)^n$. 

{\rm (3)} If $c_1<0$, then 
\begin{align*}
(-1)^n\chi\left(X_n(\underline{d}),K^\frac{k}{N}\right) & \geqslant\dbinom{n+1-\frac{N-k}{N}c_1}{n+1}+(-1)^n\dbinom{n+1-\frac{k}{N}c_1}{n+1}, & \text{if~} 0\leqslant \frac{k}{N}\leqslant \frac{1}{2};\\
\chi\left(X_n(\underline{d}),K^\frac{k}{N}\right) & 
=0, \text{~if $\frac{k}{N}=\frac{1}{2}$ and $n$ is odd}; \\
\chi\left(X_n(\underline{d}),K^\frac{k}{N}\right) & \geqslant\dbinom{n+1-\frac{k}{N}c_1}{n+1}+(-1)^n\dbinom{n+1-\frac{N-k}{N}c_1}{n+1}, & \text{if~} \frac{1}{2}\leqslant\frac{k}{N}\leqslant1. 
\end{align*}
\end{theorem}

Note that, for the Euler-Poincar\'e characteristic $\chi(M,K^{\frac{k}{N}})$ of a compact complex $n$-dimensional manifold $M$, 
\begin{enumerate}
\item if $k=0$, $\chi(M)$ is the Todd genus of $M$; 
\item if $\frac{k}{N}=\frac{1}{2}$, $\chi(M,K^{\frac{1}{2}})$ is the $\hat{A}$-genus of $M$;
\item $\chi(M,K^{1-\frac{1}{k}})$ corresponds to the $A_k$-genus up to a factor $k^n$. 
\end{enumerate}
In particular, we have 
\begin{theorem}\label{mainAkgci}
If $k\geqslant2$ and $\dfrac{c_1}{k}$ is integral, the $A_k$-genus of $X_n(\underline{d})$ satisfies the following properties:

{\rm (1)} ~$
A_k(X_n(\underline{d}))=
\begin{cases}
0, & \text{if~} c_1>0;\\
k^n\left(1+(-1)^n\right), & \text{if~} c_1=0.
\end{cases}$

{\rm (2)} If $c_1<0$, then 
$A_k(X_n(\underline{d}))\geqslant k^n\left[\dbinom{n+1-\frac{k-1}{k}c_1}{n+1}+(-1)^n\dbinom{n+1-\frac{1}{k}c_1}{n+1}\right].$
\end{theorem}
\begin{remark}
In fact, for $A_2$-genus, the corresponding power series 
\[
Q(x)=\dfrac{2x\cdot e^x}{e^{2x}-1}
\]
 is an even power series, then $A_2$-genus is expressible in Pontrjagin numbers and hence defined for an oriented smooth manifold (\cite[\S1.6]{HiBeJu1992}). Moreover, the characteristic power series corresponding to $\hat{A}$-genus is 
\[
 Q(x)=\dfrac{x/2}{\sinh(x/2)}, 
\]
 then for a compact oriented $2n$-dimensional smooth manifold $M$ with $p(TM)=(1+x_1^2)\cdots(1+x_n^2)$, we have
\begin{align*}
A_2(M)& =\prod_{i=1}^n\frac{2x_i\cdot e^{x_i}}{e^{2x_i}-1}[M]\\
&=2^n\cdot\prod_{i=1}^n\frac{x_i\cdot e^{\frac{x_i}{2}}}{e^{x_i}-1}[M]\\
&=2^n\cdot\prod_{i=1}^n\frac{\frac{x_i}{2}}{\sinh\left(\frac{x_i}{2}\right)}[M] \\
&=2^n\cdot\hat{A}(M).
\end{align*}
\end{remark}

This paper is organized as follows: In \Cref{sec-reivew}, we give a brief review of some classical Hirzebruch genera and smooth projective complete intersection. The Todd genus of a complete intersection and the proof of the main \Cref{mainTgci} are in \Cref{sec-Tg}; The elliptic genera of level $N$ and the Euler-Poincar\'e characteristic of a complete intersection are discussed in \Cref{sec-EgLN}. 
As a special case of the elliptic genera of level  $N$, the $A_k$-genus of a complete intersection and the \Cref{mainAkgci} are also discussed in \Cref{sec-EgLN}. In \Cref{sec-proof}, we give a proof of \Cref{mainNgci}, which is the main part of \Cref{Intro-mainNgci}. 

\section{Preliminaries}\label{sec-reivew}

\subsection{Hirzebruch genus}

Let $M$ be a compact $2n$-dimensional almost complex manifold. The multiplicative sequence with $Q(x)=1+b_1x+b_2x^2+\cdots$ as the characteristic power series is denoted by $\{K_j(c_1,\dots,c_j)\}$  (\cite[\S 1.8]{HiBeJu1992}), where $K_j(c_1,\dots,c_j)$ is a certain rational homogeneous polynomial of degree $2j$ in the Chern classes of $M$.
Define the characteristic class of $TM$ by
\[
K(TM)=\sum_{j=0}^\infty{K_j(c_1,\dots,c_j)}.
\]
The {\it $K$-genus of an almost complex manifold} $M$ is  the evaluation of $K(TM)$ on the fundamental class:
 \[
K(M)=K(TM)[M]=K_n(c_1,\dots,c_n)[M].
 \]
Moreover, if we consider a formal factorization of the total Chern class $c(TM)=(1+x_1)\cdots(1+x_n)$, then
 \[
K(M)=\prod_{i=1}^nQ(x_i)[M].
 \]

For a compact oriented smooth $2n$-manifold  $M$,
given an even power series $Q(x)=1+a_2x^2+a_4x^4+\cdots$, there is a multiplicative sequence $\left\{K_{j}(p_{1}, \ldots, p_{j})\right\}$ (\cite[\S 1.6]{HiBeJu1992}), where $K_j(p_{1}, \ldots, p_{j})$ is a rational homogeneous polynomial of degree $4j$  in the Pontrjagin classes.
The characteristic class $K(TM)$ is defined as 
\begin{equation*}
K(TM)=\sum_{j=0}^\infty{K_j(p_1,\dots,p_j)}.
\end{equation*}
Let $p(TM)=(1+x_1^2)\cdots(1+x_n^2)$ be the formal factorization of the total Pontrjagin class, then the {\it $K$-genus of an oriented smooth manifold} $M$ is the evaluation
\[
K(M)=K(TM)[M]=\prod_{i=1}^nQ(x_i)[M].
\]
Note that,  we put $K(M)=0$ if $2\nmid n$.

For example, we list some genera, the associated characteristic power series,  and multiplicative sequences in \Cref{CPS-Genus} (\cite{Hi1978, HiBeJu1992, Kricever1976}).

\begin{table}[h]\renewcommand\arraystretch{2.3}
\caption{characteristic power series and genus}
\begin{tabular}{l|llll}
  \toprule
$K$-genus  & $K(M)$ & $K$-class & \makecell[l]{characteristic\\ power series} & \makecell[l]{multiplicative\\ sequence}   \\
  \midrule
Todd genus & $\Td(M)$ & $\td(TM)$ & $\dfrac{x}{1-e^{-x}}$ & $\{T_j(c_1,\dots,c_j)\}$ \\  \hline
\makecell[l]{generalized\\ Todd genus} & $\Td_y(M)$ & $\td_y(TM)$ & $\dfrac{x(1+ye^{-x(1+y)})}{1-e^{-x(1+y)}}$ & $\{T_j(y;c_1,\dots,c_j)\}$  \\ \hline
$A_k$-genus & $A_k(M)$ & $A_k(TM)$ &  $\dfrac{kxe^x}{e^{kx}-1} ~(k\geqslant 2)$ & $\{(A_k)_j(c_1,\dots,c_j)\}$ \\ \hline
signature & $\sign(M)$ & $L(TM)$ &  $\dfrac{x}{\tanh x}$ (even) & $\{L_j(p_1,\dots,p_j)\}$ \\  \hline
$\hat{A}$-genus & $\hat{A}(M)$ &  $\hat{A}(TM)$ &  $\dfrac{\frac{x}{2}}{\sinh \left(\frac{x}{2}\right)}$ (even) & $\{\hat{A}_j(p_1,\dots,p_j)\}$ \\
  \bottomrule
 \end{tabular}
\label{CPS-Genus}
\end{table}

As a generalization of elliptic genus, the  elliptic genera of level $N$ were introduced independently by Hirzebruch \cite{Hi1988} and Witten \cite{Witten1987}. Furthermore, Witten conjectured its rigidity for a complex manifold with $c_1\equiv 0 \pmod N$ under circle actions, which was proved by Hirzebruch \cite{Hi1988} and Krichever \cite{Krichever1990}.

Let $\mathbb{H}=\{z\in\cc\mid {\rm Im}(z)>0\}$ be the upper half-plane of the complex numbers, and let $N$ be a natural number $>1$. For $\tau\in\mathbb{H}$ and $L=2\pi{\rm i}(\mathbb{Z}\tau+\mathbb{Z})$, there exists an elliptic function $h(x)$ with respect to $x$ for the lattice $L$ with divisor $N\cdot(0)-N\cdot(\alpha)$. Then $\alpha$ is a nonzero $N$-division point of $L$. More precisely, $\alpha=2\pi{\rm i}\dfrac{k\tau+l}{N}\neq0$ for $k,l\in\mathbb{Z}$.

Assume that $h(x)=x^N+~\text{higher terms}$. Let $f(x)=\sqrt[N]{h(x)}=x+~\text{higher terms}$. Then $f$ is elliptic with respect to a sublattice $L^\prime$ of index $N$ in $L$. For fixed $k$ and $l$, the characteristic power series $Q(x)=\dfrac{x}{f(x)}$ is uniquely determined.

For a compact almost complex manifold $M$ of real dimension $2n$, we consider a formal factorization of the total Chern class $c(M)=(1+x_1)\cdots(1+x_n), \ c_i\in H^{2i}(M;\mathbb{Z})$,
the genus $\varphi_{N,\alpha}$ associated to the power series $Q(x)$ is defined by
\begin{equation*}
\varphi_{N,\alpha}(M)=\left(\prod\limits_{i=1}^n\frac{x_i}{f(x_i)}\right)[M],
\end{equation*}
which is called the {\it elliptic genus of level $N$} for the almost complex manifold $M$. The elliptic genus of level $N$ for a compact almost complex manifold of dimension $2n$ is a modular form of weight $n$ on the congruence subgroup
\begin{equation*}
\Gamma=\left\{\begin{pmatrix} a & b \\ c & d \end{pmatrix}\in SL_2(\mathbb{Z}) \Bigg| (k~~~~l)\equiv (k~~~~l) \begin{pmatrix} a & b \\ c & d \end{pmatrix}\pmod  N\right\}.
\end{equation*}
In other words, as in \cite[page 145]{HiBeJu1992}, define the theta function by
\begin{equation*}
\Phi(\tau,x)=\left(e^\frac{x}{2}-e^{-\frac{x}{2}}\right)\prod\limits_{m=1}^\infty\frac{(1-q^me^x)(1-q^me^{-x})}{(1-q^m)^2},
\end{equation*}
where $q=e^{2\pi{\rm i}\tau}$. Then the characteristic power series corresponding to the elliptic genus of level $N$ is given by
\begin{equation*}
Q(\tau,x)=xe^{-\frac{k}{N}x}\cdot\frac{\Phi(\tau,x-\alpha)}{\Phi(\tau,x)\Phi(\tau,-\alpha)}.
\end{equation*}
In particular, if $k=0$ and $(l,N)=1$, i.e. $l$ and $N$ are coprime, then it is a modular form of weight $d$ on the congruence subgroup 
\[
\Gamma_1(N):=\left\{\begin{pmatrix} a & b \\ c & d \end{pmatrix}\in SL_2(\mathbb{Z}) \Bigg| \begin{pmatrix} a & b \\ c & d \end{pmatrix}\equiv\begin{pmatrix} 1 & b \\ 0 & 1 \end{pmatrix}\pmod  N\right\}.
\]

Unless otherwise stated, the elliptic genera of level $N$ mentioned in the remainder of this paper are modular forms on $\Gamma_1(N)$. 

\subsection{Complete intersection}

A {\it complete intersection}
$X_n(d_1,\dots,d_r)\subset \cp^{n+r}$ is a compact complex $n$-dimensional manifold given by a transversal intersection of $r$ nonsingular hypersurfaces in
the complex projective space $\cp^{n+r}$. The unordered $r$-tuple $\underline{d}:=(d_1,\dots,d_r)$
is called the {\it multi-degree}, which denotes the degrees of the $r$ nonsingular hypersurfaces.
It is well-known that the diffeomorphism type of the real $2n$-dimensional underlying manifold of $X_n(d_1,\dots,d_r)$ depends only on the multi-degree and dimension. We always identify  $X_n(d_1,\dots,d_r,1)$ with $X_n(d_1,\dots,d_r)$, since the effect of omitting the degree $1$ is to lower the embedding $X_n(d_1,\dots,d_r,1)\subset \cp^{n+r+1}$ to $X_n(d_1,\dots,d_r)\subset \cp^{n+r}$. If the multi-degree $\underline{d}=(d_1,\dots,d_r)$ is clear from the context, we set $X_n(\underline{d}):=X_n(d_1,\dots,d_r)$.
By the Lefschetz hyperplane section theorem, the inclusion $X_n(\underline{d}) \subset \mathbb{C}P^{n+r}$ is $n$-connected.

Let $\gamma$ be the pullback of the hyperplane bundle over $\cp^{n+r}$, and $x=c_1(\gamma)\in H^2 (X_n(\underline{d});\zz)\cong \zz$ be the generator. The evaluation $x^n[X_n(\underline{d})]$ equals the product of degrees: $d_1d_2\cdots d_r$, which is called the {\it total degree} of $X_n(\underline{d})$.
The total Chern class and total Pontrjagin class of $X_{n}(\underline{d})$ are given by \cite[\S 7]{Libgober&Wood1982}:
\begin{align*}
c(X_{n}(\underline{d})) & =(1+x)^{n+r+1} \cdot \prod_{i=1}^{r}\left(1+d_{i} \cdot x\right)^{-1},\\
p(X_{n}(\underline{d})) & =\left(1+x^{2}\right)^{n+r+1} \cdot \prod_{i=1}^{r}\left(1+d_{i}^{2} \cdot x^{2}\right)^{-1}.
\end{align*}
Particularly, the first Chern class is
\begin{equation}\label{ChernPontrjagin}
c_1(X_n(\underline{d}))=\left(n+r+1-\sum_{i=1}^{r}d_i\right)x.
\end{equation}
The Todd class, $A_k$-class and $\hat{A}$-class of $TX_n(\underline{d})$ are listed as follows:
\begin{align*}
\td(TX_{n}(\underline{d})) & =\left(\frac{x}{1-e^{-x}}\right)^{n+r+1}\cdot\prod_{i=1}^{r}\left(\frac{d_{i}x}{1-e^{-d_{i}x}}\right)^{-1},\\
A_k(TX_{n}(\underline{d}))& =\left(\frac{kx\cdot e^x}{e^{kx}-1}\right)^{n+r+1}\cdot\prod_{i=1}^{r}\left(\frac{kd_{i}x\cdot e^{d_{i} x}}{e^{kd_{i}x}-1}\right)^{-1},\\
\hat{A}(TX_{n}(\underline{d}))& =\left(\frac{x\cdot e^{\frac{x}{2}}}{e^{x}-1}\right)^{n+r+1}\cdot\prod_{i=1}^{r}\left(\frac{d_{i}x\cdot e^{\frac{d_{i}x}{2}}}{e^{d_{i}x}-1}\right)^{-1}.
\end{align*}

For each complete intersection, the Euler characteristic, signature, $\hat{A}$-genus, Todd genus and $A_k$-genus can be described by the dimension and multi-degree. For more details, please see \cite{WYW2020}.

Ewing and Moolgavkar showed that the Euler characteristic of $X_n(\underline{d})$ satisfies\\ $(-1)^n \chi(X_{n}(\underline{d}))\geqslant0$ for total degree $>2$ (\cite{EWingMoolgavkar1976}). Moreover, only the Euler characteristic of $X_n(2,2)$ for odd $n$ and $X_1(3)$ vanish.
If the Euler characteristic of $X_{n}(\underline{d})$ is not equal to $0$, $1$ or $-1$, then  Chen gave the following formula about $\chi(X_{n}(\underline{d}))$ (\cite{Chen1978}):
\[
(-1)^{n} \chi(X_{n}(\underline{d}))=d_{1} d_{2} \cdots d_{r}\cdot h_{n}\left[d_{1}, d_{2}, \ldots, d_{r}\right].
\]
where $h_{n}\left[d_{1}, d_{2}, \ldots, d_{r}\right]>1 \text{~for~} r \geqslant 2$,
except $n=1, r=2,
d_{1} \leqslant 2, d_{2} \leqslant 3$,
or $n=1, r=3, d_{1}, d_{2}, d_{3} \leqslant 2$.

Let $X_{2m}(\underline{d})$ be an even dimensional complete intersection, Libgober proved that its signature is monotone as a function on the $d_{i}$'s (\cite{Libgober1980}): increasing when $m$ is even and decreasing when $m$ is odd. Moreover, $(-1)^m\sign (X_{2m}(\underline{d}))>5$ except for:
\begin{align*}
& \sign(\cp^{2m})=1;\\
& \sign(X_{2m}(2))=
\begin{cases}
0, & \text{ $m$ odd};\\
2, & \text{ $m$ even}.
\end{cases}\\
& \sign(X_2(2,2))=-4.
\end{align*}

Brooks showed that, for a spin complete intersection $X_{2m}(\underline{d})$, its $\hat{A}$-genus vanishes if and only if $c_1=2m+r+1-\sum\limits_{i=1}^rd_i> 0$  (\cite{Br1983}). In \cite[Theorem 3.1]{WYW2020},  an explicit formula of the $\hat{A}(X_{2m}(\underline{d}))$ is given, and a new proof of $\hat{A}$-genus vanishing under $c_1>0$ can be found in \cite[Theorem 4.8]{WYW2020}.

Just inspired by the above-mentioned properties of Euler characteristic, signature and $\hat{A}$-genus, this paper aims to detect the similar properties of Todd genus, elliptic genera of level $N$ and $A_k$-genera for complete intersections.

\section{Todd genus of complete intersections}\label{sec-Tg}

For the $\chi_y$-characteristic of a compact complex manifold, by Atiyah-Singer index theorem, when $y=0$, the $\chi_0$-characteristic is just the Todd genus. 
By \cite[Corollary, page 160]{Hi1978}, the $\chi_{y}$-characteristic of the line bundle $\gamma^{k}$ over $X_{n}(\underline{d})$ is given by
\begin{equation}\label{Hi-cor-chiy}
\sum_{n=0}^{\infty} \chi_{y}\left(X_{n}(\underline{d}), \gamma^{k}\right) z^{n+r}=\frac{(1+z y)^{k-1}}{(1-z)^{k+1}}\cdot \prod_{i=1}^{r} \frac{(1+z y)^{d_{i}}-(1-z)^{d_{i}}}{(1+z y)^{d_{i}}+y(1-z)^{d_{i}}},
\end{equation}
where $\underline{d}=(d_1,\dots,d_r)$. So 
\begin{equation}\label{Hi-cor-p160}
\sum\limits_{n=0}^\infty \Td(X_n(\underline{d}))z^{n+r}=\frac{1}{1-z}\cdot\prod_{i=1}^r \left(1-(1-z)^{d_i}\right).
\end{equation}
Then the Todd genus of complete intersection $X_n(\underline{d})$ is the  coefficient of  $z^{n+r}$ in
\[
(1-z)^{-1}\cdot\prod_{i=1}^r\left(1-(1-z)^{d_i}\right),
\]
or equivalently the  coefficient of $z^{n+r}$ in
\begin{equation}\label{Toddgenus=coeffixn+r}
(-1)^n\cdot(1+z)^{-1}\cdot\prod_{i=1}^r\left((1+z)^{d_i}-1\right).
\end{equation}
Then it is easy to get the following result:
\begin{proposition}\label{Toddgenusci}
The Todd genus of complete intersection $X_n{(\underline{d})}$ is
\begin{equation*}
\Td(X_n(\underline{d})=\sum_{j=0}^r{{(-1)^{n+r+j}}\sum_{1\leqslant{k_1}<\dots<{k_j}\leqslant r}\dbinom{-1 +d_{k_1}+\dots+d_{k_j}}{n+r}}.
\end{equation*}
where $\dbinom{a}{k}:=\displaystyle\frac{a(a-1)(a-2)\cdots(a-k+1)}{k!}$ denotes the generalized binomial coefficients and $d_{k_1} +\dots +d_{k_j}$ vanishes if $j =0$.
\end{proposition}

\Cref{mainTgci} (1) can be directly obtained from \Cref{Toddgenusci}.
To prove \Cref{mainTgci} (2), we firstly prove two lemmas for a preparation.

\begin{lemma}\label{Tghs}
For complex hypersurface $X_n(d_1)$ with $c_1=n+2-d_1<0$, its Todd genus satisfies the following inequality:
\[
(-1)^n\Td(X_n(d_1))\geqslant n+1.
\]
\end{lemma}
\begin{proof}
By \Cref{Toddgenusci},
\[
\Td(X_n(d_1))=(-1)^{n+1}\cdot\left[\dbinom{-1}{n+1}-\dbinom{-1+d_1}{n+1}\right].
\]
Thus,
\[
(-1)^n\cdot \Td(X_n(d_1))=-\dbinom{-1}{n+1}+\dbinom{-1+d_1}{n+1}.
\]
Since $c_1=n+2-d_1<0$,
\begin{align*}
(-1)^n\cdot \Td(X_n(d_1)) & \geqslant (-1)^n+\dbinom{-1+n+3}{n+1} \\
& = (-1)^n+n+2 \\
& \geqslant n+1. \qedhere
\end{align*}
\end{proof}

Denote the multi-degree $(d_1,\dots,d_{j-1},d_{j+1},\dots,d_r)$ by $\underline{d}_{\hat{j}}$. Hence, by \eqref{Hi-cor-p160}, 
\begin{align}\label{Tdznromit}
\sum\limits_{n=0}^\infty \Td(X_n(\underline{d}))z^{n}=& \left(\frac{1}{1-z}\cdot\prod_{i=1}^{r-1}\frac{1-(1-z)^{d_i}}{z}\right)\cdot\frac{1-(1-z)^{d_r}}{z} \nonumber\\
=& \left(\sum\limits_{l=0}^\infty \Td(X_l(\underline{d}_{\hat{r}}))z^{l}\right)\cdot\frac{1-(1-z)^{d_r}}{z}.
\end{align}
Compare the coefficients of $z^{n}$ in the left and right sides of \eqref{Tdznromit}, we have
\[
\Td(X_n(\underline{d}))=\sum_{l=0}^n(-1)^{n-l}\dbinom{d_r}{n-l+1}\cdot\Td(X_l(\underline{d}_{\hat{r}})).
\]
Thus,
\begin{equation}\label{gf}
(-1)^n \Td(X_n(\underline{d}))=\sum\limits_{l=0}^n\dbinom{d_r}{n-l+1}\cdot (-1)^l \Td(X_l(\underline{d}_{\hat{r}})).
\end{equation}

\begin{lemma}\label{Tgcitohs}
For any complete intersection $X_n(d_1,\dots,d_r)$ with $r>1$, $d_1,\dots,d_r>1$ and $c_1=n+r+1-\sum\limits_{i=1}^rd_i<0$, set $c_{1}^{(n,\hat{j})}:=n+(r-1)+1-\sum\limits_{i=1}^rd_i+d_j$. If $c_{1}^{(n,\hat{j})}\geqslant 0$ for any $1\leqslant j\leqslant r$, then
\[
(-1)^n\Td(X_n(d_1,\dots,d_r))\geqslant n+r.
\]
\end{lemma}
\begin{proof}
By \Cref{Toddgenusci},
\begin{align*}
\Td(X_n(d_1,\dots,d_r))=& \sum\limits_{j=0}^r(-1)^{n+r+j}\sum\limits_{1\leqslant k_1<\cdots<k_j\leqslant r}\dbinom{-1+d_{k_1}+\cdots+d_{k_j}}{n+r}.
\end{align*}
Since $c_{1}^{(n,\hat{j})} \geqslant 0$ for $1\leqslant j\leqslant r$, then $-1+\sum_{i=1}^rd_i-d_j\leqslant n+r-1$ for $1\leqslant j\leqslant r$.  Hence,
\begin{align*}
\Td(X_n(d_1,\dots,d_r)) & =  (-1)^{n+r}\dbinom{-1}{n+r}+(-1)^{n+r+r}\dbinom{-1+d_1+\cdots+d_r}{n+r} \\
& = 1+(-1)^n\dbinom{-1+d_1+\cdots+d_r}{n+r}.
\end{align*}
Thus,
\begin{align*}
(-1)^n\Td(X_n(d_1,\dots,d_r)) & = (-1)^n+\dbinom{-1+d_1+\cdots+d_r}{n+r} \\
& \geqslant -1+n+r+1\\
& =n+r.  \qedhere
\end{align*}
\end{proof}

For complete intersection $X_n(d_1,\dots,d_{r-1},d_r)$, assume that the degrees satisfy that $d_1\geqslant\cdots \geqslant d_{r-1}\geqslant d_r>1$.
\begin{proof}[Proof of \Cref{mainTgci}]
(2) When $r=1$, \Cref{mainTgci} (2) holds due to \Cref{Tghs}.

When $r>1$, we discuss \Cref{mainTgci} (2) in two cases.

{\bf Case 1}: If $c_{1}^{(n,\hat{r})}=n+(r-1)+1-\sum\limits_{i=1}^{r-1}d_i<0$, then $c_{1}^{(l,\hat{r})}=l+(r-1)+1-\sum\limits_{i=1}^{r-1}d_i<0$ for $0\leqslant l\leqslant n$. By induction on $r$, suppose that
\[
(-1)^l \Td(X_l(d_1,\dots,d_{r-1}))\geqslant l+r-1,  \ 0\leqslant l\leqslant n.
\]
For any $0\leqslant l\leqslant n$, $\dbinom{d_r}{n-l+1}\geqslant 0$, so by \eqref{gf},
\begin{align*}
(-1)^n \Td(X_n(d_1,\dots,d_{r-1},d_r)) & \geqslant d_r\cdot(-1)^n \Td(X_n(d_1,\dots,d_{r-1})) \\
& \geqslant 2(n+r-1)\geqslant n+r.
\end{align*}

{\bf Case 2}: If $c_{1}^{(n,\hat{r})}=n+(r-1)+1-\sum\limits_{i=1}^{r-1}d_i\geqslant0$, it implies that
\[
c_{1}^{(n,\hat{j})}\geqslant0, ~\text{if $1\leqslant j\leqslant r$}.
\]
It follows from  \Cref{Tgcitohs} that
\[
(-1)^n \Td(X_n(d_1,\dots,d_{r-1},d_r))\geqslant n+r.
\]
Above all, we complete the proof  of \Cref{mainTgci} (2).
\end{proof}

\section{Elliptic genus of level $N$ for complete intersections}\label{sec-EgLN}

Firstly, let's introduce some basic results and properties on the values of elliptic genera of level $N$ at the cusps, where $N$ is a positive integer. 

\begin{theorem}\cite[\S7.2]{HiBeJu1992}
Let $M$ be a compact complex $n$-dimensional manifold,
the value of the elliptic genus of level $N$ for $M$ at the cusps of $\varGamma_1(N)$ are:
\begin{enumerate}
\item $\chi_y(M)/(1+y)^n$, where $-y=e^{2\pi{\rm i}\frac{l}{N}}\neq1$ is an $N$-th root of unity with $0<l<N$;
\item $\chi\left(M,K^\frac{k}{N}\right)$ with $0<k<N$,
\end{enumerate}
where $\chi(M,K^r)$ is the genus associated to the power series 
$Q(x)=e^{-rx}\cdot\dfrac{x}{1-e^{-x}}$,
and $\chi_y(M)=\sum\limits_{p=0}^{n}\chi^p(M)\cdot y^p$
with $\chi^p(M)=\sum\limits_{q=0}^{n}(-1)^q\cdot h^{p,q}$ is the $\chi_y$-genus introduced in \cite[\S 5.4]{HiBeJu1992}.
\end{theorem}
Note that we can formally consider the genus $\chi(M,K^r)$ for $r\in\qq$, which is in general no longer integral.
\begin{proposition}\label{prop-twopEgLN}
For any compact complex $n$-dimensional manifold $M$, we have
\begin{enumerate}
\item $\chi_{y^{-1}}(M)/(1+y^{-1})^{n}=(-1)^{n}\cdot\chi_y(M)/(1+y)^{n}$.
\item $\chi\left(M,K^{1-\frac{k}{N}}\right)=(-1)^{n}\cdot\chi\left(M,K^{\frac{k}{N}}\right)$.
\end{enumerate}
\end{proposition}
\begin{proof}
The Atiyah-Singer index theorem implies the following equality (\cite[\S 5.4]{HiBeJu1992}):
\begin{equation*}
\chi_{y^{-1}}(M) =T_n(y^{-1};c_1,\dots,c_{n})[M].
\end{equation*}
According to the arguments in \cite[page 15]{Hi1978}), 
\[
T_n(y^{-1};c_1,\dots,c_{n})[M]=\left((-y)^{-n}\cdot T_{n}(y;c_1,\dots,c_{n})\right)[M].
\]
So
\begin{align*}
\chi_{y^{-1}}(M)/(1+y^{-1})^{n} & =\left((1+y^{-1})^{-n}\cdot(-y)^{-n}\cdot T_n(y;c_1,\dots,c_{n})\right)[M] \\
& =\left(\frac{(-1)^{n}}{(1+y)^{n}}\cdot T_n(y;c_1,\dots,c_{n})\right)[M] \\
& =(-1)^{n}\chi_y(M)/(1+y)^{n}.
\end{align*}

(2) For a compact complex manifold $M$, the Hirzebruch-Riemann-Roch theorem yields:
\begin{align*}
\chi\left(M,K^{1-\frac{k}{N}}\right)& =\left({\rm ch}\left(K^{1-\frac{k}{N}}\right)\cdot\td(TM)\right)[M] \\
& =\left(e^{(1-\frac{k}{N})\cdot c_1(K)}\cdot\td(TM)\right)[M] \\
& =\left(e^{(\frac{k}{N}-1)\cdot c_1(M)}\cdot\td(TM)\right)[M].
\end{align*}
The equation (12) in \cite[\S 1.7]{Hi1978} implies that 
\[
\td(TM)=e^{\frac{1}{2}c_1(M)}\cdot\hat{A}(TM). 
\]
So
\[
\chi\left(M,K^{1-\frac{k}{N}}\right) =\left(e^{(\frac{k}{N}-\frac{1}{2})\cdot c_1(M)}\cdot\hat{A}(TM)\right)[M].
\]
Since $\hat{A}(TM)$ is a polynomial on Pontrjagin classes, all monomials in $\hat{A}(TM)$ belong to $H^m(M;\mathbb{Z})$ with $m\equiv 0 \pmod 4$. Moreover, the degree $2n-m$ part in $e^{(\frac{k}{N}-\frac{1}{2})\cdot c_1(M)}$ coincides with the one in $(-1)^{n}\cdot e^{(\frac{1}{2}-\frac{k}{N})\cdot c_1(M)}$. Thus
\begin{align*}
\chi\left(M,K^{1-\frac{k}{N}}\right) & =(-1)^{n}\cdot \left(e^{(\frac{1}{2}-\frac{k}{N})\cdot c_1(M)}\cdot\hat{A}(TM)\right)[M] \\
& =(-1)^{n}\cdot \chi\left(M,K^{\frac{k}{N}}\right).  \qedhere
\end{align*}
\end{proof}

For complete intersection $X_n(\underline{d})$ with $\underline{d}=(d_1,\dots,d_r)$, the canonical bundle of $X_n(\underline{d})$ is $K=\gamma^{-c_1}$, where $\gamma$ is the pullback of the hyperplane bundle over $\cp^{n+r}$ and $c_1=n+r+1-\sum\limits_{i=1}^rd_i$ is the coefficient of the first Chern class of $X_n(\underline{d})$ as in \eqref{ChernPontrjagin}.
Then  by \eqref{Hi-cor-chiy},  we have  
\begin{theorem} The  Euler-Poincar\'e characteristic of a complete intersection is
\begin{align} \label{EgNformula}
\chi\left(X_n(\underline{d}),K^\frac{k}{N}\right)=\sum_{j=0}^r(-1)^{n+r+j}\sum_{1\leqslant k_1<\dots<k_j\leqslant r}\dbinom{\frac{k}{N}c_1-1+d_{k_1}+\cdots+d_{k_j}}{n+r}.
\end{align}
Moreover, if $c_1\equiv 0\pmod N$,  $\chi\left(X_n(\underline{d}),K^{\frac{k}{N}}\right)$ is integral.
\end{theorem}

Based on \eqref{EgNformula}, we can prove the following theorem, which is the main part of \Cref{Intro-mainNgci}.
\begin{theorem}\label{mainNgci}
For complete intersection $X_n(\underline{d})$ with $c_1=n+r+1-\sum\limits_{i=1}^rd_i$, if $c_1\equiv 0 \pmod N$, then $\chi\left(X_n(\underline{d}),K^\frac{k}{N}\right)$ with $0\leqslant k\leqslant N$ satisfy the following properties:

{\rm (1)} If $c_1>0$,  $
\chi\left(X_n(\underline{d}),K^\frac{k}{N}\right)=
\begin{cases}
0, & \text{if~ $0<k<N$} ; \\
1,  & \text{if~ $k=0$} ; \\
(-1)^n,  & \text{if~ $k=N$}.
\end{cases}$

{\rm (2)} If $c_1=0$,  $\chi\left(X_n(\underline{d}),K^\frac{k}{N}\right)=1+(-1)^n$.

{\rm (3)} If $c_1<0$ and  $\frac{1}{2}\leqslant \frac{k}{N}\leqslant 1$,  
\[
\chi\left(X_n(\underline{d}),K^\frac{k}{N}\right)\geqslant \dbinom{n+1-\frac{k}{N}c_1}{n+1}+(-1)^n\dbinom{n+1-\frac{N-k}{N}c_1}{n+1}.
\]
\end{theorem}
By combining \Cref{mainNgci} and \Cref{prop-twopEgLN},  \Cref{Intro-mainNgci} holds. 
The proof of \Cref{mainNgci} will be presented in \Cref{sec-proof}.

Note that for $k=N$, the condition $c_1\equiv 0\pmod N$ is unnecessary in the proof of \Cref{mainNgci}. Moreover, by \Cref{prop-twopEgLN},
\[
\chi(X_n(\underline{d}),K)=(-1)^n\chi(X_n(\underline{d})). 
\]
By the Hirzebruch-Riemann-Roch theorem, 
\[
\chi(X_n(\underline{d}))=\Td(X_n(\underline{d})), 
\]
thus we have
\begin{corollary}\label{cor-todd}
For any complete intersection $X_n(\underline{d})$,  the Todd genus of $X_n(\underline{d})$ satisfies the following properties: 

{\rm (1)} ~$
\Td(X_n(\underline{d}))=
\begin{cases}
1, & \text{if~} c_1>0;\\
1+(-1)^n, & \text{if~} c_1=0.
\end{cases}$

{\rm (2)} If $c_1<0$, then $
(-1)^n\Td\left(X_n(\underline{d})\right)\geqslant \dbinom{n+1-c_1}{n+1}+(-1)^n.$
\end{corollary}

Observe that, under  $c_1<0$, $(-1)^n\Td(X_n(\underline{d}))$ has different lower bounds in \Cref{cor-todd} and \Cref{mainTgci}.

Since $\chi(M,K^{\frac{1}{2}})=\hat{A}(M)$, \Cref{mainNgci} implies that 

\begin{corollary}\label{cor-hatA}
For any complete intersection $X_n(\underline{d})$ with $c_1$ even, the $\hat{A}$-genus of $X_n(\underline{d})$ satisfies the following properties:
\begin{align*}
\hat{A}(X_n(\underline{d})) & =
\begin{cases}
0, & \text{if $ c_1>0$ or $n$ is odd};\\
1+(-1)^n, & \text{if~} c_1=0; 
\end{cases}\\
\hat{A}(X_n(\underline{d})) & \geqslant 2\dbinom{n+1-\frac{1}{2}c_1}{n+1}, \text{if $c_1<0$ and $n$ is even}.
\end{align*}
\end{corollary}
For $A_k$-genus ($k\geqslant 2$), the corresponding characteristic power series is
\[
Q(x)=\frac{kx\cdot e^x}{e^{kx}-1}.
\]
By the calculations in \cite[Example 5.8]{WYW2020} or \cite[Corollary 2, page 794]{Kricever1976}, the $A_k$-genus of complete intersection $X_n(\underline{d})$ is the coefficient of $z^{n+r}$ in
\begin{equation}\label{Akgenus=coeffixn+r}
k^n(1+z)^{\frac{c_1}{k}}\cdot(1+z)^{-1}\cdot\prod_{i=1}^r\left((1+z)^{d_i}-1\right),
\end{equation}
where $c_1=n+r+1-\sum\limits_{i=1}^r d_i$. More precisely, we have 
\begin{proposition}\label{Akgenusformula}
The $A_k$-genera {\rm ($k\geqslant 2$)} of $X_n(\underline{d})$ is
\begin{equation*}
A_k(X_n(\underline{d}))=k^n\sum_{j=0}^r(-1)^{r-j}\sum_{1\leqslant k_1<\dots<k_j\leqslant r}\dbinom{\frac{c_1}{k}-1+d_{k_1}+\cdots+d_{k_j}}{n+r}.
\end{equation*}
\end{proposition}
Note that, the $A_1$-genus is just the Todd genus.
Comparing \eqref{Toddgenus=coeffixn+r} and \eqref{Akgenus=coeffixn+r}, it is easy to get the following property.
\begin{proposition}
The $A_k$-genus of $X_n(\underline{d})$ and the Todd genus of $X_l(\underline{d})$ with $0\leqslant l\leqslant n$ has the following relation:
\[
A_k(X_n(\underline{d}))=k^n\left[\sum\limits_{l=0}^n\dbinom{\frac{c_1}{k}}{n-l}\cdot(-1)^l \Td(X_l(\underline{d}))\right].
\]
\end{proposition}

Let $M$ be a compact complex manifold of dimension $n$.
Since the power series corresponding to the $A_k$-genus satisfies that
\[
Q(x)=\frac{kx\cdot e^x}{e^{kx}-1}=\frac{kx}{1-e^{-kx}}e^{-(k-1)x}, k\geqslant 2,
\]
by \cite[Appendix III]{HiBeJu1992}, 
\[
A_k(M)=k^n\chi\left(M,K^{\frac{k-1}{k}}\right).
\]
Hence, we get the following result:
\begin{theorem}
For complete intersection $X_n(\underline{d})$ with $c_1=n+r+1-\sum\limits_{i=1}^rd_i$. If $k\geqslant 2$ and $c_1\equiv 0 \pmod k$,  the $A_k$-genus of $X_n(\underline{d})$ satisfies the following properties:

{\rm (1)} ~$
A_k(X_n(\underline{d}))=
\begin{cases}
0, & \text{if~} c_1>0;\\
k^n(1+(-1)^n), & \text{if~} c_1=0.
\end{cases}$

{\rm (2)} If $c_1<0$, then $
A_k(X_n(\underline{d}))\geqslant k^n\left[\dbinom{n+1-\frac{k-1}{k}c_1}{n+1}+(-1)^n\dbinom{n+1-\frac{1}{k}c_1}{n+1}\right].$
\end{theorem}

\begin{remark}\label{rem-nomore}
The results in this paper yields no more information about the $S^1$-action on complete intersections viewed as complex manifolds.  For example, only under $c_1<0$, the automorphism group of holomorphic transformation on $X_n(\underline{d})$ is finite (\cite[Chapter III Theorem 2.1]{Kobayashi}), then there exists no non-trivial $S^1$-action preserving the complex structure on $X_n(\underline{d})$.
Moreover, as shown in \cite[Theorem 3.1]{Benoist2013}, for $X_n(\underline{d})$ with $n\geqslant 2$, the automorphism group is zero-dimensional except for the quadric $X_n(2)$ and the complex projective spaces. In other words only the homogeneous complete intersections admit a non-trivial circle action preserving the complex structure (\cite{DessaiWiemeler}). In contrast, there is relatively little known results about the smooth $S^1$-action on complete intersections. For more details please see \cite{DessaiWiemeler}.
\end{remark}

\section{Proof of \Cref{mainNgci}}\label{sec-proof}

\begin{proof}[Proof of \Cref{mainNgci}]
(1) If $c_1=n+r+1-\sum\limits_{i=1}^r d_i>0$ and $c_1\equiv 0\pmod N$, when $0< k<N$,  for any $0\leqslant j\leqslant r$ and $1\leqslant k_1\leqslant \cdots\leqslant k_j\leqslant r$, we have
\[
0\leqslant \frac{k}{N}c_1-1+d_{k_1}+\cdots+d_{k_j}<n+r.
\]
Hence by \eqref{EgNformula},  for $0< k<N$,
\[
\chi\left(X_n(\underline{d}),K^\frac{k}{N}\right)=\sum\limits_{j=0}^r(-1)^{n+r+j}\sum\limits_{1\leqslant k_1<\dots<k_j\leqslant r}\dbinom{\frac{k}{N}c_1-1+d_{k_1}+\cdots+d_{k_j}}{n+r}=0.
\]

When $k=0$, the only nonzero generalized binomial coefficient in \eqref{EgNformula} is $\dbinom{-1}{n+r}=(-1)^{n+r}$, which holds corresponding to $j=0$. So $\chi\left(X_n(\underline{d})\right)=1$. 

When $k=N$,
the result $\chi\left(X_n(\underline{d}), K\right)=(-1)^n$ follows directly from \Cref{prop-twopEgLN}. 

(2) If $c_1=n+r+1-\sum\limits_{i=1}^r d_i=0$, then by \Cref{Toddgenusci},
\[
\chi\left(X_n(d_1,\dots,d_r),K^\frac{k}{N}\right)=\Td(X_n(d_1,\dots,d_r))=1+(-1)^n.
\]

(3) In the case of $c_1=n+r+1-\sum\limits_{i=1}^r d_i<0$, we mainly use the induction on the number of degrees to prove 
\[
\chi\left(X_n(d_1,\dots,d_r),K^\frac{k}{N}\right)\geqslant \dbinom{n+1-\frac{k}{N}c_1}{n+1}+(-1)^n\dbinom{n+1-\frac{N-k}{N}c_1}{n+1}.
\]

{\bf Case (3)-1:} $r=1$. For the hypersurface $X_n(d_1)$ with $c_1=n+2-d_1<0$ and $\dfrac{k}{N}c_1$ being integral, note that 
\[
(-1)^k\dbinom{a}{k}=\dbinom{k-1-a}{k}, 
\]
then by \eqref{EgNformula},
\begin{align*}
\chi\left(X_n(d_1),K^\frac{k}{N}\right)& =(-1)^{n+1}\left[\dbinom{\frac{k}{N}c_1-1}{n+1}-\dbinom{\frac{k}{N}c_1-1+d_1}{n+1}\right] \\
& =(-1)^{n+1}\left[(-1)^{n+1}\dbinom{n+1-\frac{k}{N}c_1}{n+1}-\dbinom{\frac{k}{N}c_1-1+d_1}{n+1}\right] \\
& =\dbinom{n+1-\frac{k}{N}c_1}{n+1}+(-1)^n\dbinom{\frac{k}{N}c_1-1+d_1}{n+1} \\
& =\dbinom{n+1-\frac{k}{N}c_1}{n+1}+(-1)^n\dbinom{n+1-\frac{N-k}{N}c_1}{n+1}. 
\end{align*}

{\bf Case (3)-2:} $r=2$. For $X_n(d_1,d_2)$, the intersection of two hypersurfaces, if $c_1=n+3-d_1-d_2<0$, then by \eqref{EgNformula},
\begin{align}
\chi\left(X_n(d_1,d_2),K^\frac{k}{N}\right) & =(-1)^{n+2}\left[\dbinom{\frac{k}{N}c_1-1}{n+2}-\dbinom{\frac{k}{N}c_1-1+d_1}{n+2}\right. \nonumber\\
&~~ \left.-\dbinom{\frac{k}{N}c_1-1+d_2}{n+2}+\dbinom{\frac{k}{N}c_1-1+d_1+d_2}{n+2}\right] \nonumber\\
& =\dbinom{(1-\frac{k}{N})c_1-1+d_1+d_2}{n+2}-\dbinom{(1-\frac{k}{N})c_1-1+d_2}{n+2} \nonumber\\
&~~ -\dbinom{(1-\frac{k}{N})c_1-1+d_1}{n+2}+\dbinom{(1-\frac{k}{N})c_1-1}{n+2}. \label{Akg2degree}
\end{align}

\noindent {\bf Claim:}
For any integers $a,b,n$, let $n>0$, $a\geqslant b$ and $a+b\geqslant n-1$, then
\begin{equation}\label{4dbinom}
\dbinom{a+1}{n}+\dbinom{b-1}{n}\geqslant \dbinom{a}{n}+\dbinom{b}{n}.
\end{equation}
The equality holds if $a+b=n-1$ and $n$ is odd. 
\begin{proof}[Proof of \eqref{4dbinom}]
\[
\dbinom{a+1}{n}+\dbinom{b-1}{n}-\dbinom{a}{n}-\dbinom{b}{n}=\dbinom{a}{n-1}-\dbinom{b-1}{n-1}.
\]

If $b>0$, $\dbinom{a}{n-1}-\dbinom{b-1}{n-1}\geqslant 0$. In addition, if $a+b=n-1$ and $n$ is odd,  $\dbinom{a}{n-1}-\dbinom{b-1}{n-1}= 0$

If $b\leqslant 0$, then $a\geqslant n-1$ and
\begin{align*}
\dbinom{a+1}{n}+\dbinom{b-1}{n}-\dbinom{a}{n}-\dbinom{b}{n}& =\dbinom{a}{n-1}-(-1)^{n-1}\dbinom{n-1-b}{n-1}\\
& \geqslant \dbinom{a}{n-1}-\dbinom{n-1-b}{n-1}\\
& \geqslant 0. 
\end{align*}
In addition, if $a+b=n-1$ and $n$ is odd, we have 
\[
\dbinom{a}{n-1}-(-1)^{n-1}\dbinom{n-1-b}{n-1}=\dbinom{a}{n-1}-\dbinom{n-1-b}{n-1}=0. \qedhere
\]
\end{proof}
Without loss of generality, let's assume  $d_1\geqslant d_2\geqslant2$. Note that in the right side of \eqref{Akg2degree}, when $\frac{k}{N}\geqslant\frac{1}{2}$,  
\[
\left(1-\frac{k}{N}\right)c_1-1+d_1+\left(1-\frac{k}{N}\right)c_1-1+d_2=\left(1-\frac{2k}{N}\right)c_1+n+1\geqslant n+1.
\]
Using \eqref{Akg2degree} and \eqref{4dbinom},  it implies that
\[
\chi\left(X_n(d_1,d_2),K^\frac{k}{N}\right)\geqslant \chi\left(X_n(d_1+1,d_2-1),K^\frac{k}{N}\right).
\]
Thus, for any positive integer $e<d_2$, we have
\[
\chi\left(X_n(d_1,d_2),K^\frac{k}{N}\right)\geqslant \chi\left(X_n(d_1+e,d_2-e),K^\frac{k}{N}\right).
\]
Hence,
\begin{align*}
\chi\left(X_n(d_1,d_2),K^\frac{k}{N}\right) & \geqslant \chi\left(X_n(d_1+d_2-1,1),K^\frac{k}{N}\right)\\
& = \dbinom{n+1-\frac{k}{N}c_1}{n+1}+(-1)^n\dbinom{n+1-\frac{N-k}{N}c_1}{n+1},
\end{align*}
where $c_1=n+3-d_1-d_2$.

{\bf Case (3)-3:} For $r>2$, without loss of generality, let's assume that $d_1\geqslant d_2\geqslant\cdots\geqslant d_r>1$.

For any $p, q, e$ satisfying $1\leqslant p<q\leqslant l<r$,  $0<e<d_q$,
suppose that 
\begin{align}
& \chi\left(X_n(d_1,\dots,d_p,\dots,d_q,\dots,d_l),K^\frac{k}{N}\right)\geqslant \chi\left(X_n(d_1,\dots,d_p+e,\dots,d_q-e,\dots,d_l),K^\frac{k}{N}\right),\label{inducAkdegreel} \\
& \chi\left(X_n(d_1,\dots,d_l),K^\frac{k}{N}\right)\geqslant \dbinom{n+1-\frac{k}{N}c_1}{n+1}+(-1)^n\dbinom{n+1-\frac{N-k}{N}c_1}{n+1}, \label{inducAkdegreel-0}
\end{align}
where $c_1=n+l+1-\sum\limits_{i=1}^ld_i$.

It is clear that \eqref{inducAkdegreel} and \eqref{inducAkdegreel-0} hold for $l=2$. We aim to show that \eqref{inducAkdegreel-0} holds for $l=r$.
Note that for the genus of $\chi\left(X_n(d_1,\dots,d_r),K^\frac{k}{N}\right)$, the $f(x) =e^{\frac{k}{N}x}\cdot(1-e^{-x})$, which occurs in the corresponding characteristic power series $Q(x)=\dfrac{x}{f(x)}$, satisfies the identity
\begin{align}\label{AkR(u)}
& f(u_1)f(u_2)f(u_3)-f(u_1+w)f(u_2-w)f(u_3)\nonumber \\
= ~& f(u_1)f(w)f(u_2-w+u_3)-f(w)f(u_2-w)f(u_1+u_3).
\end{align}
By the virtual genus discussed in \cite[page 36]{HiBeJu1992} or \cite[\S9, \S11]{Hi1978},
\eqref{AkR(u)} implies that
\begin{align*}
& \chi\left(X_n(d_1,d_2,d_3,d_4,\dots,d_r),K^\frac{k}{N}\right)-\chi\left(X_n(d_1+e,d_2-e,d_3,d_4,\dots,d_r),K^\frac{k}{N}\right) \\
=~& \chi\left(X_n(d_1,e,d_2-e+d_3,d_4,\dots,d_r),K^\frac{k}{N}\right)-\chi\left(X_n(e,d_2-e,d_1+d_3,d_4,\dots,d_r),K^\frac{k}{N}\right).
\end{align*}
In particular, for $e=1$,
\begin{align}
& \chi\left(X_n(d_1,d_2,d_3,d_4,\dots,d_r),K^\frac{k}{N}\right)-\chi\left(X_n(d_1+1,d_2-1,d_3,d_4,\dots,d_r),K^\frac{k}{N}\right) \nonumber \\
=~& \chi\left(X_n(d_1,d_2-1+d_3,d_4,\dots,d_r),K^\frac{k}{N}\right)-\chi\left(X_n(d_2-1,d_1+d_3,d_4,\dots,d_r),K^\frac{k}{N}\right).  \label{Akdiff4}
\end{align}
By the induction hypothesis \eqref{inducAkdegreel}, the right side of \eqref{Akdiff4} is non-negative. Hence,
\begin{align*}
\chi\left(X_n(d_1,d_2,d_3,d_4,\dots,d_r),K^\frac{k}{N}\right) & \geqslant
\chi\left(X_n(d_1+1,d_2-1,d_3,d_4,\dots,d_r),K^\frac{k}{N}\right) \\
& \geqslant ~\cdots\\
& \geqslant \chi\left(X_n(d_1+d_2-1,1,d_3,d_4,\dots,d_r), K^\frac{k}{N}\right)\\
& \geqslant \dbinom{n+1-\frac{k}{N}c_1}{n+1}+(-1)^n\dbinom{n+1-\frac{N-k}{N}c_1}{n+1}.
\end{align*}
We have completed the induction, and the \Cref{mainNgci} holds.
\end{proof}

\end{document}